\documentclass[10pt,a4paper]{article}
\usepackage{amsmath}
\usepackage{amsthm}
\usepackage{amsfonts}
\usepackage{latexsym}

\newcounter{alphthm}
\newtheorem{thm}{Theorem}[section]
\newtheorem{lem}[thm]{Lemma}
\newtheorem{cor}{Corollary}[section]

\theoremstyle{definition}

\newcommand{\be}{\begin{equation}}
\newcommand{\ee}{\end{equation}}
\newcommand{\ben}{\begin{enumerate}}
\newcommand{\een}{\end{enumerate}}

\title{Finsler Metrics with Bounded Cartan Torsions}
\author{A. Tayebi, H. Sadeghi and   E. Peyghan}
\begin{document}

\maketitle
\begin{abstract}
The norm of Cartan torsion plays an important role for studying of immersion  theory in Finsler geometry. Indeed,
Finsler manifold with unbounded Cartan torsion can not be isometrically imbedded into any Minkowski space. In this
paper, we find two subclasses of $(\alpha, \beta)$-metrics which have  bounded Cartan torsion.  Then, we give two
subclasses of $(\alpha, \beta)$-metrics whose bound on the Cartan torsions are independent of the norm of $\beta$.\\\\
{\bf {Keywords}}:    Cartan Torsion, $(\alpha, \beta)$-metric,  Randers metric.\footnote{ 2010 Mathematics subject
Classification: 53C60, 53C25.}
\end{abstract}

\section{Introduction}
One of fundamental problems in Finsler geometry is wether or not every Finsler manifold  can be isometrically
immersed into a Minkowski space.   The answer is affirmative for Riemannian manifolds \cite{N}. For Finsler
manifolds, the problem under some conditions was considered by   Burago-Ivanov, Gu  and Ingarden  \cite{BI}
\cite{Gu1}\cite{Gu2}\cite{In}.
In \cite{Sh3}, Shen proved that Finsler manifold with unbounded Cartan torsion can not be isometrically imbedded
into any Minkowski space.  Thus the norm of Cartan torsion plays a important role for studying of immersion
theory in Finsler geometry. For a Finsler manifold $(M, F)$, the second and third order derivatives of
${1\over 2} F_x^2$ at $y\in T_xM_0$ are fundamental form ${\bf g}_y$  and the Cartan torsion ${\bf C}_y$
on $T_xM$, respectively \cite{AS}. The Cartan torsion  was first introduced by Finsler \cite{F} and emphased
by Cartan \cite{C}.  For the Finsler metric $F$, one can defines the norm of  the Cartan torsion ${\bf C}$ as follows
\[
||{\bf C}||=\sup_{F(y)=1, v\neq 0}\frac{|{\bf C}_y(v, v, v)|}{[{\bf g}_y(v, v)]^{\frac{3}{2}}}.
\]
The bound for two dimensional Randers metrics $F=\alpha+\beta$ is verified by Lackey \cite{BCS}. Then, Shen proved
that the Cartan torsion of Randers metrics  on a manifold $M$ of dimension $n\geq 3$ is uniformly bounded by
$3/\sqrt{2}$ \cite{Sh1}. In \cite{MZ}, Mo-Zhou extend his result to a  general Finsler metrics,
$F=\frac{(\alpha+\beta)^m}{\alpha^{m-1}}$ ($m\in [1, 2]$). Recently, the first two authors find a
relation between the norm of Cartan and mean Cartan torsions of Finsler metrics defined by a Riemannian
metric and a 1-form on a manifold \cite{TS}. They prove that generalized Kropina metrics
$F=\frac{\alpha^{m+1}}{\beta^m}$, $(m\neq  0)$  have  bounded Cartan torsion.
It turns out that every C-reducible Finsler metric has bounded Cartan torsion.

All of above metrics are special Finsler metrics so- called $(\alpha,\beta)$-metrics. An $(\alpha, \beta)$-metric
is a Finsler metric on $M$ defined by $F:=\alpha\phi(s)$, where $s=\beta/\alpha$,  $\phi=\phi(s)$ is a $C^\infty$
function on the $(-b_0, b_0)$ with certain regularity, $\alpha=\sqrt{a_{ij}(x)y^iy^j}$ is a Riemannian metric and
$\beta=b_i(x)y^i$ is a 1-form on $M$.

In this paper, we consider a special $(\alpha,\beta)$-metric, called the generalized Randers metric
$F=\sqrt{c_1\alpha^2+2c_2\alpha\beta+c_3\beta^2}$ on a manifold $M$. By putting $c_1=c_2=c_3=1$,  we get
the Randers metric. First, we prove the following.

\begin{thm}\label{mainthm1}
Let $F=\sqrt{c_1\alpha^2+2c_2\alpha\beta+c_3\beta^2}$ be the generalized Randers metric on a manifold $M$,
where $\alpha=\sqrt{a_{ij}(x)y^iy^j}$ is a Riemannian metric, $\beta= b_i(x)y^i$ is an 1-form on $M$ and
$c_1, c_2, c_3$ are real constants such that $c_2^2<c_1c_3$ and $c_1^2>|c_2(3c_1+c_3)|$. Then $F$ has bounded Cartan torsion.
\end{thm}

\bigskip

One of important $(\alpha,\beta)$-metrics is Berwald metric which was introduced by L. Berwald on unit ball
$U=B^n$ \cite{Be}. Berwald's metric can be expressed in the form $F=\lambda(\alpha+\beta)^2/\alpha$, where
\begin{equation}\label{Funk}
\alpha=\frac{\sqrt{(1-|x|^2)|y|^2-<x,y>^2}}{1-|x|^2}, \ \  \beta=\frac{<x,y>}{1-|x|^2}, \ \ \lambda=\frac{1}{1-|x|^2},
\ \ y\in T_xB^n\simeq \mathbb{R}^n\nonumber
\end{equation}
and $<,>$ and $| . |$ denote the Euclidean inner product and norm on $\mathbb{R}^n$, respectively. The Berwald's metric
 has been generalized by Shen to an arbitrary convex domain $U\subset R^n$ \cite{SH}. As an extension of the Berwald
 metric, we consider the metric $F=c_1\alpha+c_2\beta+c_3\beta^2/\alpha$ where
 $c_1, c_2, c_3\in \mathbb{R}$. Then we prove the following.

\begin{thm}\label{mainthm2}
Let $F=c_1\alpha+c_2\beta+c_3\beta^2/\alpha$ be an $(\alpha,\beta)$-metric on a manifold $M$, where
$\alpha:=\sqrt{a_{ij}(x)y^iy^j}$ is a Riemannian metric, $\beta:= b_i(x)y^i$ is an 1-form on $M$ and $c_1, c_2, c_3$
are real constants such that $c_2^2<4c_1c_3$ and $|c_1|>|c_3|$. Then $F$ has bounded Cartan torsion.
\end{thm}

For a vector $y \in T_xM_{0}$, the Riemann curvature $R_y: T_xM \rightarrow T_xM$ is defined by
$R_y(u):=R^i_{\ k}(y)u^k {\partial \over {\partial x^i}}$, where
\[
R^i_{\ k}(y)=2G^i_{x^k}- G^i_{x^jy^k}y^j+2G^j G^i_{y^j y^k}-G^i_{y^j}G^j_{y^k}.
\]
where $G^i=\frac{1}{4}g^{il}[(F^2)_{x^k y^l}y^k-(F^2)_{x^l}]$ are called the spray coefficients.
The family $R:=\{R_y\}_{y\in TM_0}$ is called the Riemann curvature. There are many Finsler metrics whose Riemann
curvature in every direction is quadratic.  A Finsler metric $F$ is said to be Berwald metric and R-quadratic metric
if $G^i$ and  $R_y$ is quadratic in $y\in T_xM$ at each point $x\in M$, respectively. In \cite{Sh1}, Shen proved
that every complete R-quadratic  manifold with bounded Cartan torsion is Landsbergian. He proved that a regular
$(\alpha, \beta)$-metric is Landsbergian if and only if it is Berwaldian \cite{Shen}. Thus, we can conclude the
following.
\begin{cor}
Let  $F_1=c_1\alpha+c_2\beta+c_3\beta^2/\alpha$,  $(c_2^2<4c_1c_3,|c_1|>|c_3|)$ and $F_2=\sqrt{c_1\alpha^2+
2c_2\alpha\beta+c_3\beta^2}$, $(c_2^2<c_1c_3$, $c_1^2>|c_2(3c_1+c_3)|)$ are  R-quadratic Finsler metrics on
a complete manifold $M$.  Then $F_1$ and $F_2$ are Berwaldian.
\end{cor}

\bigskip

By Theorems \ref{mainthm1}, it follows that if $c_2^2=c_1c_3$ then the norm of Cartan torsion of Finsler metric
$F=\sqrt{c_1\alpha^2+2c_2\alpha\beta+c_3\beta^2}$ is independent of $b:=||\beta||_\alpha=\sqrt{b_ib^i}$, where
$b^i=a^{ji}b_j$.  It is an interesting problem, to find a  subclass of $(\alpha, \beta)$-metrics whose bound on
the Cartan torsion is independent of $b$. In the final section, we give two subclasses of $(\alpha, \beta)$-metrics
whose bound on the Cartan torsions are independent of $b$.

\section{Preliminaries}\label{sectionP}
Let $M$ be an $n$-dimensional $ C^\infty$ manifold. Denote by $T_x M $ the tangent space at $x \in M$,  by
$TM=\cup _{x \in M} T_x M $ the tangent bundle of $M$, and by $TM_{0} = TM \setminus \{ 0 \}$ the slit tangent
bundle of $M$. A  Finsler metric on $M$ is a function $ F:TM \rightarrow [0,\infty)$ which has the following
properties:\\
(i) $F$ is $C^\infty$ on $TM_{0}$;\\
(ii) $F$ is positively 1-homogeneous on the fibers of the tangent bundle of $M$;\\
(iii) for each $y\in T_xM$, the following quadratic form $\textbf{g}_y$ on
$T_xM$  is positive definite,
\[
\textbf{g}_{y}(u,v):={1 \over 2} \frac{\partial^2}{\partial s \partial t}\left[  F^2 (y+su+tv)\right]|_{s,t=0}, \ \
u,v\in T_xM.
\]

Let  $x\in M$ and $F_x:=F|_{T_xM}$. We define ${\bf C}_y:T_xM\otimes T_xM\otimes
T_xM\rightarrow \mathbb{R}$ by
\[
{\bf C}_{y}(u,v,w):={1 \over 2} \frac{d}{dt}\left[\textbf{g}_{y+tw}(u,v)
\right]|_{t=0}, \ \ u,v,w\in T_xM.
\]
The family ${\bf C}:=\{{\bf C}_y\}_{y\in TM_0}$  is called the Cartan torsion. It is well known that ${\bf{C}}=0$
if and only if $F$ is Riemannian.

For $y\in T_x M_0$, define  mean Cartan torsion ${\bf I}_y$ by ${\bf I}_y(u):=I_i(y)u^i$, where $I_i:=g^{jk}C_{ijk}$.
By Deicke's  theorem, $F$ is Riemannian  if and only if ${\bf I}_y=0$  \cite{D}.

\bigskip

Let $(M, F)$ be a Finsler manifold. For  $y \in T_xM_0$, define the  Matsumoto torsion
${\bf M}_y:T_xM\otimes T_xM \otimes T_xM \rightarrow \mathbb{R}$ by ${\bf M}_y(u,v,w):=M_{ijk}(y)u^iv^jw^k$, where
\[
M_{ijk}:=C_{ijk} - {1\over n+1}  \{ I_i h_{jk} + I_j h_{ik} + I_k h_{ij} \},\label{Matsumoto}
\]
and $h_{ij}:=g_{ij}-\frac{1}{F^2}g_{ip}y^pg_{jq}y^q$ is the angular metric. A Finsler metric $F$ is said to be
C-reducible, if ${\bf M}_y=0$ \cite{Mat2}.  Matsumoto proves that every Randers metric satisfies ${\bf M}_y=0$.
Later on, Matsumoto-H\={o}j\={o} prove that the converse is true too.
\begin{lem}\label{MaHo}{\rm (\cite{MH})}
\emph{A Finsler metric $F$ on a manifold of dimension $n\geq 3$ is a Randers metric if and only if the Matsumoto
torsion vanish}.
\end{lem}
Let $\alpha=\sqrt{a_{ij}(x)y^iy^j}$ be a Riemannian metric, and $\beta=b_i(x)y^i$ be a 1-form on $M$ with
$\|\beta\|=\sqrt{a^{ij}b_ib_j}<1$. The Finsler metric $F=\alpha+\beta$ is called a Randers metric, which has important
applications both in mathematics and physics.

For a Finsler metric $F=F(x,y)$ on a smooth manifold $M$, geodesic curves are characterized by the system of second
order differential equations
\[
\frac{d^2 x^i}{dt^2}+2G^i\big(x,\frac{dx}{dt}\big)=0,
\]
where the local functions $G^i=G^i(x, y)$ are called the  spray coefficients, and given by
\[
G^i(x,y):=\frac{1}{4}g^{il}\Big[\frac{\partial^2F^2}{\partial x^k \partial y^l}y^k-
\frac{\partial F^2}{\partial x^l}\Big].
\]
In  a standard local  coordinates $(x^i,y^i)$ in $TM$, the vector field
$\textbf{G}=y^i {{\partial} \over {\partial x^i}}-2G^i(x,y){{\partial} \over {\partial y^i}}$ is called the spray of $F$.
A Finsler metric $F$ is  called a Berwald metric, if  $G^i$  are quadratic in $y\in T_xM$  for any $x\in M$.
The Berwald spaces can be viewed as  Finsler spaces modeled on a single Minkowski space \cite{ShDiff}.
\section{Proof of Theorem \ref{mainthm1}}
Let us first consider the case of $dim M = 2$. There exists
a local ortonormal  coframe $\{\omega_1,\omega_1\}$ of Riemannian metric $\alpha$. So $\alpha^2$ can be wrote as
\begin{eqnarray*}
\alpha^2=\omega_1^2+\omega_2^2.
\end{eqnarray*}
If we denote $\alpha=\sqrt{a_{ij}y^iy^j}$ where $y=\sum_{i=1}^{2}y^ie_i$ and $\{e_i\}$ is the dual frame of
$\{\omega_i\}$ then $a_{ij}=\delta_{ij}$ and $a^{ij}=\delta^{ij}$. Adjust coframe $\{\omega_1,\omega_1\}$
properly such that
\begin{eqnarray*}
\beta=k\omega_1.
\end{eqnarray*}
Then $b_1=k$ and $b_2=0$ where $\beta=\sum_{i=1}^{2}b_iy^i$. Hence
\[
\|\beta\|_{\alpha}:=\sqrt{a^{ij}b_ib_j}=k.
\]
For an arbitrary tangent vector $y=ue_1+ve_2\in T_pM$, we can obtain that
\begin{eqnarray*}
&&\alpha(p,y)=\sqrt{u^2+v^2},\hspace{.5cm}\beta(p,y)=ku,\\
&&F(p,y)=\sqrt{c_1(u^2+v^2)+2c_2ku\sqrt{u^2+v^2}+c_3k^2u^2}.
\end{eqnarray*}
Assume that $y^{\perp}$ satisfies:
\begin{eqnarray}
\textbf{g}_y(y,y^{\perp})=0, \ \ \ \textbf{g}_y(y^{\perp},y^{\perp})=F^2(p,y).\label{c0}
\end{eqnarray}
Obviously $y^{\perp}$ is unique because the metric is non-degenerate. The frame $\{y,y^{\perp}\}$ is
called the Berwald frame.

Let
\begin{eqnarray*}
y=r\cos(\theta)e_1+r\sin(\theta)e_2
\end{eqnarray*}
i.e.
\begin{eqnarray*}
u=r\cos(\theta),\hspace{.5cm}v=r\sin(\theta).
\end{eqnarray*}
Plugging the above expression into (\ref{c0}) and computing by Maple program (see
Section \ref{sec1}) yields
\begin{eqnarray}
y^{\perp}=\frac{r\big(-\sin(\theta)(c_2k\cos(\theta)+c_1),c_3k^2\cos(\theta)+c_2k+c_1\cos(\theta)+
c_2k\cos(\theta)^2\big)}{\sqrt{c_1c_3k^2+c_2c_3k^3\cos(\theta)^3+3c_2^2k^2\cos(\theta)^2
+3c_1c_2k\cos(\theta)-c_2^2k^2+c_1^2}}.\label{c2}
\end{eqnarray}
By the definition of the bound of Cartan torsion, it is easy to show that for the Berwald frame $\{y,y^{\perp}\}$,
\begin{eqnarray*}
\|\textbf{C}\ \|_p=\sup_{y\in T_pM_0}\xi(p,y),
\end{eqnarray*}
where
\begin{eqnarray*}
\xi(p,y):=\frac{F(p,y)|\textbf{C}_y(y^{\perp},y^{\perp},y^{\perp})|}{|\textbf{g}_y(y^{\perp},y^{\perp})|^{\frac{3}{2}}}.
\end{eqnarray*}
Again computing by Maple program (see Subsection \ref{sub1} below), we obtain
\begin{eqnarray*}
\xi(p,y)=\frac{3}{2}\big|\frac{c_2k\sin(\theta)\big(c_1+2c_2k\cos(\theta)+c_3k^2\cos(\theta)^2\big)^2}
{(c_1c_3k^2+c_2c_3k^3\cos(\theta)^3+3c_2^2k^2\cos(\theta)^2+3c_1c_2k\cos(\theta)-c_2^2k^2+c_1^2)^{\frac{3}{2}}
}\big|
\end{eqnarray*}
Define two functions on $[0,1]\times [-1,1]$ by following
\begin{eqnarray*}
&&f(k,x):=c_1c_3k^2+c_2c_3k^3x^3+3c_2^2k^2x^2+3c_1c_2kx-c_2^2k^2+c_1^2,\\
&&g(k,x):=\frac{3}{2}\frac{c_2k\sqrt{1-x^2}\big(c_1+2c_2kx+c_3k^2x^2\big)^2}{f(k,x)^{\frac{3}{2}}}.
\end{eqnarray*}
Hence
\begin{eqnarray}
\|\textbf{C}\ \|_p=\max_{0\leq\theta\leq 2\pi}|g(k,\cos\theta)|.\label{c3}
\end{eqnarray}
For a fixed $k=k_0$ $(k_0\in [0,1])$, we have
\begin{eqnarray*}
\frac{\partial}{\partial x}f(k_0,x)=3c_1c_2k_0+6c_2^2k_0^2x+3c_2c_3k_0^3x^2.
\end{eqnarray*}
So from $\frac{\partial}{\partial x}f(k_0,x)=0$,  we have
\begin{eqnarray*}
x\in\Big\{\frac{-c_2+\sqrt{c_2^2-c_1c_3}}{c_3k_0},  \frac{-c_2-\sqrt{c_2^2-c_1c_3}}{c_3k_0}\Big\}.
\end{eqnarray*}
Because of $c_2^2<c_1c_3$,  we conclude that $f(k_0,x)$ is ascending
 or descending. So for $x\in [-1,1]$ we have
\[
f(k_0,x)\geq \min\{f(k_0,-1),f(k_0,1)\}.
\]
By a simple computation, we have
\begin{eqnarray*}
&&f(k_0,1)=c_1^2+3c_1c_2k_0+c_1c_3k_0^2+2c_2^2k_0^2+c_2c_3k_0^3,\\
&&f(k_0,-1)=c_1^2-3c_1c_2k_0+c_1c_3k_0^2+2c_2^2k_0^2-c_2c_3k_0^3.
\end{eqnarray*}
Since $c_1c_3\geq 0$ and $c_1^2>|c_2(3c_1+c_3)|$, then we have
\[
f(k_0,1)> 0,\ \ \  f(k_0,-1)> 0.
\]
So for $k\in [0,1]$ and $x\in [-1,1]$, we have $f(k,x)> 0$. Thus $g(k,x)$ is continuous in $[0,1]\times [-1,1]$
 and has a upper bound.

In higher dimensions, the definition of the Cartan torsion's bound at $p\in M$ is
 \begin{eqnarray*}
 \|\textbf{C}\ \|_p=\sup_{y,u\in T_pM}\frac{F(p,y)|\textbf{C}_y(u,u,u)|}{|\textbf{g}_y(u,u)|^{\frac{3}{2}}}.
\end{eqnarray*}
Considering the plane $P=span\{u,y\},$ from the above conclusion we obtain $\|\textbf{C}\ \|_p$ is bounded.
Furthermore, the bound is independent of the plane $P\subset T_pM$ and the point $p\in M.$
Hence the Cartan torsion is also bounded. This completes the proof.
\section{Proof of Theorem \ref{mainthm2}}
In this section, we are going to prove the Theorem \ref{mainthm2}. Let us first consider the case of $dim M = 2$.
 By the similar method used in proof of Theorem \ref{mainthm1}, for an arbitrary tangent vector $y=ue_1+ve_2\in T_pM$
 we can obtain that
\begin{eqnarray*}
&&\alpha(p,y)=\sqrt{u^2+v^2},\hspace{.5cm}\beta(p,y)=ku\\
&&F(p,y)=c_1\sqrt{u^2+v^2}+c_2ku+c_3\frac{k^2u^2}{\sqrt{u^2+v^2}}.
\end{eqnarray*}
Using the Maple program (see Section \ref{sec2}), we get
\begin{eqnarray}
y^{\perp}=\frac{r\big(-\sin(\theta)(c_1-c_3k^2\cos(\theta)^2),c_1\cos(\theta)+kc_2+2c_3k^2\cos(\theta)-
c_3k^2\cos(\theta)^3\big)}{\sqrt{(-3k^2c_3\cos(\theta)^2+2k^2c_3+c_1)(k^2c_3\cos(\theta)^2+
kc_2\cos(\theta)+c1)}}\label{c2}
\end{eqnarray}
Again computing by Maple program (see subsection \ref{sub2}), we obtain
\begin{eqnarray*}
\xi(p,y)=\frac{3}{2}\big|\frac{k\sin(\theta)\big(-c_1c_2-4k^3c_3^2\cos(\theta)
+8c_3^2k^3\cos(\theta)^3-2k^2c_2c_3+5c_2c_3k^2\cos(\theta)^2\big)}{(3c_3k^2\cos(\theta)^2-2c_3k^2-c1)
\sqrt{(-3c_3k^2\cos(\theta)^2+2c_3k^2+c_1)(c_1+c_2k\cos(\theta)+c_3k^2\cos(\theta)^2)}
}\big|.
\end{eqnarray*}
Define three functions on $[0,1]\times [-1,1]$
\begin{eqnarray*}
&&f_1(k,x):=3c_3k^2x^2-2c_3k^2-c_1,\\
&&f_2(k,x):=c_1+c_2kx+c_3k^2x^2,\\
&&g(k,x):=\frac{3}{2}\frac{k\sqrt{1-x^2}\big(-c_1c_2-4k^3c_3^2x+8c_3^2k^3x^3-2k^2c_2c_3+5c_2c_3k^2x^2\big)}
{f_1(k,x)\sqrt{-f_1(k,x)f_2(k,x)}}.
\end{eqnarray*}
Hence
\begin{eqnarray}
\|\textbf{C}\|_p=\max_{0\leq\theta\leq 2\pi}|g(k,\cos\theta)|.\label{c3}
\end{eqnarray}
For a fixed $k=k_0$ $(k_0\in [0,1])$ the roots  of  $f_2(k_0,x)$ are
 \begin{eqnarray*}
\Big\{\frac{1}{2}\frac{-c_2+\sqrt{c_2^2-4c_1c_3}}{c_3k_0},\frac{1}{2}\frac{-c_2-\sqrt{c_2^2-4c_1c_3}}{c_3k_0}\Big\}.
\end{eqnarray*}
Because of $c_2^2<4c_1c_3$, for $x\in [-1,1]$ and $k\in [0,1]$ we have $f_2(k_0,x)\neq 0$. If $c_3\geq 0$ because
of $c_2^2<4c_1c_3$ we get $c_1\geq 0$ and the maximum of $f_1(k,x)$ in $[-1,1]$ occurred in $x\in\{-1,1\}$.
By simple computation we have:
 \begin{eqnarray}
 f_1(k_0,-1)=f_1(k_0,1)=c_3k_0^2-c_1.\label{c4}
\end{eqnarray}
By the assumption, we have $|c_1|>|c_3|$ so we conclude that
\[
f_1(k_0,-1)=f_1(k_0,1)<0.
\]
If $c_3\leq 0$ because of $c_2^2<4c_1c_3$ we have $c_1\leq 0$ and the minimum of $f_1(k,x)$ in $[-1,1]$ occurred
in $x\in\{-1,1\}$. By the assumption, we have $|c_1|>|c_3|$ so by $ (\ref{c4})$ we conclude that
\[
f_1(k_0,-1)=f_1(k_0,1)>0.
\]
So for $k\in [0,1]$ and $x\in [-1,1]$, we have 
\[
f_1(k,x)\neq 0.
\]
Then $g(k,x)$ is continuous in $[0,1]\times [-1,1]$
and has a upper bound. For the higher dimensions, proof is similar to the 2-dimensional case.
\section{Maple Programs}
In this section, we are give the Maple programs  which  used  to proving the  Theorems \ref{mainthm1} and
Theorem \ref{mainthm2}.
\subsection{ Berwald Frame}
The special and useful Berwald frame was introduced and developed by Berwald. Let $(M, F)$ be a two-dimensional
Finsler manifold. We study two dimensional Finsler space and define a local field of orthonormal frame
$(\ell^i,m^i)$ called the Berwald frame, where  $\ell^i=y^i/F(y)$, $m^i$ is the unit vector with $\ell_i m^i=0$,
$\ell_i=g_{ij}\ell^i$ and $g_{ij}$ is  defined by $g_{ij}=\ell_i\ell_j+m_im_j$.
\subsubsection{Berwald Frame of $F=\sqrt{c_1\alpha^2+2c_2\alpha\beta+c_3\beta^2}$}\label{sec1}
$> restart;$\\
$> with(linalg):$\\
$> F:=sqrt\big((c[1])*(u^{\wedge}2+v^{\wedge}2)+2*(c[2])*k*u*sqrt(u^{\wedge}2+v^{\wedge}2)
+(c[3])*k^{\wedge}2*u^{\wedge}2\big):$\\
$> g:=simplify(1/2*hessian(F^{\wedge}2,[u,v])):$\\
$> gr:=simplify(subs(u=cos(theta),v=sin(theta),g)):$\\
$> y:=vector(2,[r*cos(theta),r*sin(theta)]);$
\begin{eqnarray*}
y:=[u=r \cos(\theta), v=r \sin(\theta)]
\end{eqnarray*}
$> yp:=vector(2):$\\
$> eq:=simplify(evalm(transpose(y)\&*gr\&*yp))=0:$\\
$> x:=solve(eq,yp[1]);$
\begin{eqnarray*}
x=-\frac{yp_2 sin(\theta)(c_2k\cos(\theta)+c_1)}{c_3k^2\cos(\theta)+c_2k+c_1\cos(\theta)+c_2kcos(\theta)^2}
\end{eqnarray*}
$> ny:=simplify(r^{\wedge}2*subs(u=cos(theta),v=sin(theta),F^{\wedge}2));$
\begin{eqnarray*}
ny:=r^2(c_1+2c_2k\cos(\theta)+c_3k^2\cos(\theta)^2)
\end{eqnarray*}
$> yp[1]:=-\sin(theta)*(c[2]*k*\cos(theta)+c[1]):$\\
$> yp[2]:=c[3]*k^{\wedge}2*\cos(theta)+c[2]*k+c[1]*\cos(theta)+c[2]*k*\cos(theta)^{\wedge}2:$\\
$> nyp:=simplify(evalm(transpose(yp)\&*gr\&*yp)):$\\
$> lambda:=simplify(sqrt(r^{\wedge}2*nyp/ny)/r):$\\
$> yp[1]:=yp[1]/lambda:$\\
$> yp[2]:=yp[2]/lambda:$\\
$> print(yp);$
\begin{eqnarray*}
&&\Big[-\frac{\sin(\theta)(c_2k\cos(\theta)+c_1)r}{\sqrt{c_1c_3k^2+c_2c_3k^3\cos(\theta)^3
+3c_2^2k^2\cos(\theta)^2+3c_1c_2k\cos(\theta)-c_2^2k^2+c_1^2}},\\
&&\frac{(c_3k^2\cos(\theta)+c_2k+c_1\cos(\theta)+c_2k\cos(\theta)^2)r}
{\sqrt{c_1c_3k^2+c_2c_3k^3\cos(\theta)^3+3c_2^2k^2\cos(\theta)^2+3c_1c_2k\cos(\theta)-c_2^2k^2+c_1^2}}\Big]
\end{eqnarray*}
\subsubsection{Berwald Frame of $F=c_1\alpha+c_2\beta+c_3\beta^2/\alpha$}\label{sec2}
$> restart;$\\
$> with(linalg):$\\
$> F:=c[1]*sqrt(u^{\wedge}2+v^{\wedge}2)+c[2]*k*u+(c[3]*k^{\wedge}2*u^{\wedge}2)/(sqrt(u^{\wedge}2+v^{\wedge}2)):$\\
$> g:=simplify(1/2*hessian(F^{\wedge}2,[u,v])):$\\
$> gr:=simplify(subs(u=cos(theta),v=sin(theta),g)):$\\
$> y:=vector(2,[r*cos(theta),r*sin(theta)]);$
\begin{eqnarray*}
y:=[u=r \cos(\theta), v=r \sin(\theta)]
\end{eqnarray*}
$> yp:=vector(2):$\\
$> eq:=simplify(evalm(transpose(y)\&*gr\&*yp))=0:$\\
$> x:=solve(eq,yp[1]);$
\begin{eqnarray*}
x=-\frac{(c_3k^2\cos(\theta)^2-c_1)yp_2\sin(\theta)}{c_3k^2\cos(\theta)^3-2c_3k^2\cos(\theta)-c_1\cos(\theta)-c_2k}
\end{eqnarray*}
$> ny:=simplify(r^{\wedge}2*subs(u=cos(theta),v=sin(theta),F^{\wedge}2));$
\begin{eqnarray*}
ny:=r^2(c_1+c_2k\cos(\theta)+c_3k^2\cos(\theta)^2)^2
\end{eqnarray*}
$> yp[1]:=-(-c[1]+c[3]*k^{\wedge}2*\cos(\theta)^{\wedge}2)*\sin(\theta):$\\
$> yp[2]:=:-c[1]*\cos(theta)-k*c[2]-2*c[3]*k^{\wedge}2*\cos(theta)+c[3]*k^{\wedge}2*\cos(theta)^{\wedge}3:$\\
$> nyp:=simplify(evalm(transpose(yp)\&*gr\&*yp)):$\\
$> lambda:=simplify(sqrt(r^{\wedge}2*nyp/ny)/r):$\\
$> yp[1]:=yp[1]/lambda:$\\
$> yp[2]:=yp[2]/lambda:$\\
$> print(yp);$
\begin{eqnarray*}
&&\big[\frac{-(-c_1+c_3k^2\cos(\theta))
r\sin(\theta)}{\sqrt{(-3k^2c_3\cos(\theta)^2+2k^2c_3+c_1)(k^2c_3\cos(\theta)^2+kc_2\cos(\theta)+c1)}}\\
&&,\frac{(-c_1\cos(\theta)-c_2k-2c_3k^2\cos(\theta)+c_3k^2\cos(\theta)^3)r}
{\sqrt{(-3k^2c_3\cos(\theta)^2+2k^2c_3+c_1)(k^2c_3\cos(\theta)^2+kc_2\cos(\theta)+c1)}}\big]
\end{eqnarray*}
\subsubsection{The Method of Computation}
\textbf{Step 1:} Solve the equation $\textbf{g}_y(y,y^{\perp})=0$.
\begin{eqnarray*}
(x,yp_{[2]})=(\frac{yp_{[2]}yp_{[1]}}{yp_{[2]}},yp_{[2]})
\end{eqnarray*}
and $yp :=(yp_{[1]}, yp_{[2]})$ is a particular solution.
\bigskip

\noindent
\textbf{Step 2:} Assume that $y^{\perp}=\frac{1}{\lambda}yp$
 is the satisfied solution. Notice that
\begin{eqnarray*}
\textbf{g}_y(y^{\perp},y^{\perp})=F^2(y):=ny
\end{eqnarray*}
Then we get
\begin{eqnarray*}
\lambda=\sqrt{\frac{nyp}{ny}}
\end{eqnarray*}
which $nyp$ is defined by
\begin{eqnarray*}
nyp:=\textbf{g}_y(yp,yp)
\end{eqnarray*}
\bigskip

\noindent
\textbf{Step 3:} Plug these results into $y^{\perp}$ , we get the Berwald frame $\{y,y^{\perp}\}$.
 \subsection{Computation of $\xi(p,y)$ }
Here, we are going to compute $\xi(p,y)$  for the Finsler metrics defined by
$F=\sqrt{c_1\alpha^2+2c_2\alpha\beta+c_3\beta^2}$ and $F=c_1\alpha+c_2\beta+c_3\beta^2/\alpha$,
where $c_1, c_2$ and $c_3$ are real numbers.
 \subsubsection{Computation of $\xi(p,y)$ for $F=\sqrt{c_1\alpha^2+2c_2\alpha\beta+c_3\beta^2}$}\label{sub1}
 $> nyp:=simplify(evalm(transpose(yp)\&*gr\&*yp));$\\
 \begin{eqnarray*}
nyp:=r^2\big(c_1+2c_2k\cos(\theta)+c_3k^2\cos(\theta)^3\big)
\end{eqnarray*}
$> bc:=factor(abs(simplify(r^{\wedge}2*subs(t=0,q=0,p=0,diff(subs(u=\cos(theta)+t*yp[1]/$\\
$> r+q*yp[1]/ r+p*yp[1]/r,$\\
$> v=\sin(theta)+t*yp[2]/r+q*yp[2]/r+p*yp[2]/r,F^{\wedge}2/4),[t,q,p])))/nyp));$
\begin{eqnarray*}
bc:=\frac{3}{2}\Big|\frac{c_2k\sin(\theta)\big(c_1+2c_2k\cos(\theta)+c_3k^2\cos(\theta)^2\big)^2}
{(c_1c_3k^2+c_2c_3k^3\cos(\theta)^3+3c_2^2k^2\cos(\theta)^2+3c_1c_2k\cos(\theta)-c_2^2k^2+c_1^2)^{\frac{3}{2}}
}\Big|
\end{eqnarray*}

\bigskip

\subsubsection{Computation of $\xi(p,y)$ for $F=c_1\alpha+c_2\beta+c_3\beta^2/\alpha$}\label{sub2}
$> nyp:=simplify(evalm(transpose(yp)\&*gr\&*yp));$\\
 \begin{eqnarray*}
nyp:=r^2(c_1^2+2c_1c_2k\cos(\theta)+2c_1c_3k^2\cos(\theta)^2+c_2^2k^2\cos(\theta)^2+2c_2c_3k^3
\cos(\theta)^3+c_3^2k^4\cos(\theta)^4)
\end{eqnarray*}
$> bc:=factor(abs(simplify(r^{\wedge}2*subs(t=0,q=0,p=0,diff(subs(u=\cos(theta)+t*yp[1]/$\\
$> r+q*yp[1]/ r+p*yp[1]/r,$\\
$> v=\sin(theta)+t*yp[2]/r+q*yp[2]/r+p*yp[2]/r,F^{\wedge}2/4),[t,q,p])))/nyp));$
\begin{eqnarray*}
bc:=\frac{3}{2}\Big|\frac{k\sin(\theta)\big(-c_1c_2-4k^3c_3^2\cos(\theta)+8c_3^2k^3\cos(\theta)^3
-2k^2c_2c_3+5c_2c_3k^2\cos(\theta)^2\big)}{(3c_3k^2\cos(\theta)^2-2c_3k^2-c1)\sqrt{(-3c_3k^2\cos(\theta)^2
+2c_3k^2+c_1)(c_1+c_2k\cos(\theta)+c_3k^2\cos(\theta)^2)}
}\Big|
\end{eqnarray*}

\subsubsection{The method of computation:}
Let
\begin{eqnarray*}
nyp:=g_y(yp,yp)=g_y(y^{\perp},y^{\perp})
\end{eqnarray*}
Then compute
\begin{eqnarray*}
bc=\frac{F(y)\textbf{C}_y(y^{\perp},y^{\perp},y^{\perp})}{g_y(y^{\perp},y^{\perp})^{\frac{3}{2}}}
\end{eqnarray*}
This is prepared for estimating the bound of Cartan torsion.

\section{Some Remarks}
In this section, we will link our  theorems to the results in \cite{TS} and discuss some related
problems. In \cite{TS}, the following is proved.
\begin{thm}\label{MainTHM2}{\rm (\cite{TS})}
\emph{Let $F=\alpha\phi(s)$ be a non-Riemannian $(\alpha,\beta)$-metric on a manifold $M$ of dimension
$n\geq3$. Then the norm of Cartan and mean Cartan torsion of $F$  satisfy in following relation
\begin{equation}\label{4}
\|{\bf C}\|=\sqrt{\frac{3p^2+6p\ q+(n+1)q^2}{n+1}}\ \|{\bf I}\|,
\end{equation}
where $p=p(x,y)$ and $q=q(x,y)$ are scalar function on $TM$ satisfying $p+q=1$ and given by following}
\begin{eqnarray}\label{1}
p\!\!\!\!&= &\!\!\!\!\ \frac{n+1}{a A}\Big[s( \phi\phi'' + \phi'\phi'  ) -  \phi \phi'\Big],\label{p}\\
a\!\!\!\!&:=&\!\!\!\! \phi (  \phi -s \phi'),\\
A\!\!\!\!&= &\!\!\!\!\ (n-2) \frac{s\phi''}{\phi-s\phi'} -(n+1) \frac{\phi'}{\phi}
-\frac{(b^2-s^2)\phi'''-3s\phi''}{(b^2-s^2)\phi''+\phi-s\phi'}.\label{A}
\end{eqnarray}
\end{thm}

\bigskip
\noindent
The  Cartan tensor of an $(\alpha,\beta)$-metric  is given by following
\begin{equation}\label{0}
C_{ijk}={\frac{p}{1+n}}\{h_{ij}I_k+h_{jk}I_i+h_{ki}I_j\}+\frac{q}{\|{\bf I}\|^2}I_iI_jI_k.
\end{equation}
where $p=p(x,y)$  and $q=q(x,y)$ are scalar functions on $TM$ satisfying $p+q=1$ and $p$ is defined by (\ref{p}).
It is remarkable that, a Finsler metric is called semi-C-reducible if its Cartan tensor is given by the equation
(\ref{0}).  It is proved that every non-Riemannian $(\alpha,\beta)$-metric on a manifold $M$ of dimension $n\geq3$
is semi-C-reducible \cite{Mat5}. A Finsler metric $F$ is said to be $C2$-like  if $p=0$ and is called C-reducible if $q=0$.

It is well-known conclusive theorem that every C-reducible Finsler
metric is Randers metric. Now, we have a natural problem: Are semi-C-reducible metrics necessarily $(\alpha,\beta)$-metric?

\begin{cor}
Let $F=\alpha\phi(s)$ be a non-Randersian  $(\alpha,\beta)$-metric on a manifold $M$ of dimension $n\geq3$. Then $F$  is not a C2-like  metric.
\end{cor}
\begin{proof}
By Theorem \ref{MainTHM2}, $F$ is a C2-like  metric if and only if $\phi$ satisfies the following
\be\label{ODE}
s ( \phi\phi'' + \phi'\phi'  ) -  \phi \phi'=0.
\ee
Solving (\ref{ODE}), we obtain that
\[
\phi=\sqrt{c_1s^2+c_2},
\]
where $c_1$ and $c_2$ are two real constant. For this $\phi$, the $(\alpha,\beta)$-metric $F=\alpha\phi(s)$ is Riemannian, which is a contradiction.
\end{proof}
Thus the converse of Theorem \ref{MainTHM2} is not true.

\bigskip
For  a generalized Randers metric $F=\sqrt{c_1\alpha^2+2c_2\alpha\beta+c_3\beta^2}$ on a manifold $M$, we
have 
\[
\phi=\sqrt{c_1+2c_2s+c_3s^2}.
\]
Then we get the following
\begin{eqnarray*}
a\!\!\!\!&=&\!\!\!\! c_1+c_2s\\
A\!\!\!\!&=&\!\!\!\! \frac{3(n-2)sk}{(c_1+c_2s)\phi^4}-\frac{(n+1)(c_2+c_3s)}{\phi^2}+
\frac{3(c_cc_3-c_2^2)s\phi^2-3(b^2-s^2)k}{\phi^2\big[(c_1+c_2s)\phi^2+(b^2-s^2)(c_cc_3-c_2^2)s\big]},
\end{eqnarray*}
where $k:=(c_2^2-c_1c_3)(c_1+c_3s)$. Thus
\begin{eqnarray}
p=-\frac{(n+1)c_2}{(c_1+c_2s)A}.\label{pgr}
\end{eqnarray}
Then we get the following.
\begin{cor}\label{corgr}
Let $F=\sqrt{c_1\alpha^2+2c_2\alpha\beta+c_3\beta^2}$ be a generalized Randers metric on a manifold $M$.
Then the relation between the norm of Cartan and mean Cartan torsion of $F$  satisfy in (\ref{4}) where $p$
is given by (\ref{pgr}).
\end{cor}

\bigskip

For the  metric $F=c_1\alpha+c_2\beta+c_3\beta^2/\alpha$, we have  
\[
\phi=c_1+c_2s+c_3s^2.
\]
Then we get the following
\begin{eqnarray*}
a\!\!\!\!&=&\!\!\!\! (c_1-c_3s^2)(c_1+c_2s+c_3s^2)\\
A\!\!\!\!&=&\!\!\!\! \frac{2(n-2)c_3s}{c_1-c_3s^2}-\frac{(n+1)(c_2+2c_3s)}{c_1+c_2s+c_3s^2}+\frac{6c_3 s}
{(c_1-c_3 s^2)+2(b^2-s^2)c_3}.
\end{eqnarray*}
Thus
\be
p=\frac{n+1}{(c_1+c_3s^2)(c_1+c_2s+c_3s^2)A}\Big[3c_2c_3s^2+4c_3s^3-c_1c_2\Big].\label{RK}
\ee
By taking $c_3=0$,  we have the Randers metric $F=c_1\alpha+c_2\beta$. In this case, by (\ref{RK}) we get $p=1$ and $q=0$.
Thus for a Randers metric, we have the following
\[
C_{ijk}={\frac{1}{1+n}}\{h_{ij}I_k+h_{jk}I_i+h_{ki}I_j\}, \ \ \textrm{and}
\ \ \|{\bf C}\|=\sqrt{\frac{3}{n+1}}\ \|{\bf I}\|.
\]
Now,  if we put $c_1=c_3=1$ and $c_2=2$ then we get Berwald metric $F=\frac{(\alpha+\beta)^2}{\alpha}$.
Similar to the Corollary \ref{corgr}, we get the following.
\begin{cor}
Let $F=c_1\alpha+c_2\beta+c_3\beta^2/\alpha$ be an $(\alpha, \beta)$-metric on a manifold $M$.
Then the relation between the norm of Cartan and mean Cartan torsion of $F$  satisfy in (\ref{4}) where
$p$ is given by (\ref{RK}).
\end{cor}

\bigskip

Now, let $c_2^2=c_1c_3$. Then
\[
k=0 \ \ \   \textrm{and} \ \  A=-\frac{(n+1)(c_2+c_3s)}{\phi^2}.
\]
In this case, it is easy to see that the norm of Cartan torsions of Finsler metric
$F=\sqrt{c_1\alpha^2+2c_2\alpha\beta+c_3\beta^2}$ is independent of $b$. It is an interesting problem, to find a
subclass of $(\alpha, \beta)$-metrics whose bound on the Cartan torsion is independent of $b=||\beta||_\alpha$.
Here, we give some Finsler metrics with such property. For this work, we  find all of solutions that for them
the numerator of final sentence in  (\ref{A}) is vanishing, i.e.,
\be
(b^2-s^2)\phi'''-3s\phi''=0.\label{ODE}
\ee
The solutions of (\ref{ODE}) are given by following
\begin{eqnarray}
\phi_1=-\frac{d_1\sqrt{s^2-b^2}}{b^2}+d_2s+d_3\label{p1}
\end{eqnarray}
and
\begin{eqnarray}
\phi_2=\frac{d_1\sqrt{b^2-s^2}}{b^2}+d_2s+d_3,\label{p2}
\end{eqnarray}
where $d_1,d_2,d_3$ are constants. Then we get the following.
\begin{thm}
Let $F=\alpha\phi(s)$ are the $(\alpha, \beta)$-metrics  defined by (\ref{p1}) or (\ref{p2}).
Then the norm of Cartan torsion of $F$ is independent of $b=||\beta||$.
\end{thm}

\bigskip
The other simple answer to this question arise when numerator of final sentence in  (\ref{A}) is a multiplying
factor of the  denominator. In this case, there is a real constant $\lambda$ such that the following holds
\be
\frac{(b^2-s^2)\phi'''-3s\phi''}{(b^2-s^2)\phi''+\phi-s\phi'}=\lambda.
\ee
Thus we have the following ODE
\be
\phi'''-\big(\lambda+\frac{3s}{b^2-s^2}\big)\phi''+\frac{s\ \lambda }{b^2-s^2}\phi'-
\frac{\lambda }{b^2-s^2}\phi=0.\label{ODE2}
\ee
The solutions of (\ref{ODE2}) are given by following
\be
\phi=c_1s+c_2\sqrt{b^2-s^2}+c_3\sqrt{b^2-s^2}\int \frac{e^{\lambda s}}{(b^2-s^2)^{\frac{3}{2}}} ds,\label{phi}
\ee
where $c_1,c_2,c_3$ are real constants. Therefore,  we have the following.
\begin{thm}
Let $F=\alpha\phi(s)$ are the $(\alpha, \beta)$-metrics  defined by (\ref{phi}). Then the norm of Cartan torsion
of $F$ is independent of $b=||\beta||$.
\end{thm}

\bigskip
\noindent
\textbf{Open Problems.} Some natural question arises as following:\\
(I) How large is the subclass of  $(\alpha, \beta)$-metrics which their norm of Cartan torsion are
independent of $b=||\beta||$?\\
(II) The other question is to find all of $(\alpha, \beta)$-metrics with bounded Cartan torsion.


\bigskip

\noindent
Akbar Tayebi and Hassan Sadeghi\\
Department of Mathematics, Faculty  of Science\\
University of Qom\\
Qom. Iran\\
Email:\ akbar.tayebi@gmail.com\\
Email:\ sadeghihassan64@gmail.com
\bigskip

\noindent
Esmaeil Peyghan\\
Department of Mathematics, Faculty  of Science\\
Arak University\\
Arak 38156-8-8349, Iran\\
Email: epeyghan@gmail.com

\end{document}